\documentclass{amsart}
\usepackage{amsmath}
\usepackage{amssymb}
\usepackage{latexsym}
\usepackage{amsbsy}
\usepackage[all]{xy}
\usepackage{amscd}
\usepackage[dvips]{graphicx}
\newtheorem{theorem}{Theorem}[section]

\theoremstyle{definition}

\newtheorem{Prop}[theorem]{Proposition}
\newtheorem{Cor}[theorem]{Corollary}

\theoremstyle{remark}

\numberwithin{equation}{section}

      \def\ud{{\underline{d}}}

  \def\bbz{{\mathbb Z}}  \def\bbq{{\mathbb Q}} \def\bb1{{\mathbb 1}}

  \def\bbc{{\mathbb C}}

\def\uq2{U_q(\hat{sl}_2)}

\def\bb{{\bf b}}

\def\nd{{\noindent}}

\def\mc{{\mathcal{C}}}

\def\ue{{\underline{e}}}

\begin{document}

\title[On the cluster multiplication formulas] {Notes on the cluster multiplication formulas for 2-Calabi-Yau categories}

\author{Ming Ding and Fan Xu}
\address{Institute for advanced study\\
Tsinghua University\\
Beijing 100084, P.~R.~China} \email{m-ding04@mails.tsinghua.edu.cn
(M.Ding)}
\address{Department of Mathematical Sciences\\
Tsinghua University\\
Beijing 100084, P.~R.~China} \email{fanxu@mail.tsinghua.edu.cn
(F.Xu)}
\thanks{Fan Xu was supported by
Alexander von Humboldt Stiftung and was also partially supported by
the Ph.D. Programs Foundation of Ministry of Education of China (No.
200800030058)}


\date{January 28, 2010.}


\begin{abstract}
Y. Palu has generalized the cluster multiplication formulas to
2-Calabi-Yau categories with cluster tilting objects (\cite{Palu2}).
The aim of this note is to construct a variant of Y. Palu's formula
and deduce a new version of the cluster multiplication formula
(\cite{XiaoXu}) for acyclic quivers  in the context of cluster
categories.
\end{abstract}
\maketitle

\section*{Introduction}
Cluster algebras were introduced by S. Fomin and A. Zelevinsky
\cite{FZ} in order to develop a combinatorial approach to study
problems of total positivity in algebraic groups and canonical bases
in quantum groups. The link between acyclic cluster algebras and
representation theory of quivers were first revealed in \cite{MRZ}.
In~\cite{BMRRT}, the authors introduced the cluster categories as
the categorification of acyclic cluster algebras. In \cite{CC}, the
authors introduced a certain structure of Hall algebra involving the
cluster category by associating its objects to some variables given
by an explicit map $X_{?}$ called the Caldero-Chapoton map. The
images of the map are called generalized cluster variables.  For
simply laced Dynkin quivers, P. Caldero and B. Keller constructed a
cluster multiplication formula (of finite type) between two
generalized cluster variables (\cite{CK2005}). The Caldero-Chapoton
map and the Caldero -Keller cluster multiplication theorem open a
way to construct cluster algebras from 2-Calabi-Yau categories. The
cluster multiplication formula of finite type was generalized to
affine type in \cite{Hubery2005} and any type in \cite{XiaoXu} and
\cite{Xu}.  Y. Palu \cite{Palu2} further extended the formula to
2-Calabi-Yau categories with cluster tilting objects.

The aim this note is twofold. One is to simplify the cluster
multiplication formula in \cite{XiaoXu} and \cite{Xu}. In practice,
the formula is useful for constructing $\bbz$-bases of cluster
algebras of affine type ( \cite{DXX}, \cite{Du}). However, the
formula is given in the context of module categories and a bit
complicate. We construct a more `unified'  version in the context of
cluster categories (Theorem \ref{unified}). The other aim is to
construct a variant of Y.Palu's formula (Theorem \ref{DX}). In
particular, when applied to acyclic quivers, the variant is exactly
the cluster multiplication formula in \cite{XiaoXu} and \cite{Xu}.
Since the variant implies Y.Palu's formula, we can view it as a
refinement of Y.Palu's formula.

\section{The cluster multiplication formulas for 2-Calabi-Yau categories}
We recall the notations used in \cite{Palu1}\cite{Palu2}. Let $k$ be
the field of complex numbers and $\mathcal{C}$ be a Hom-finite,
2-Calabi-Yau and Krull-Schmidt $k$-linear triangulated category with
a basic cluster tilting object $T$. We set
$B=\mathrm{End}_{\mc}(T,T)$ and $F=\mathcal{C}(T,-)$. Thus following
the hypotheses above, the functor $F:\mathcal{C}\longrightarrow
\mathrm{mod}\ B $ induces an equivalence of categories:
$$\mathcal{C}/(T[1])\cong
\mathrm{mod}\ B $$ where $(T[1])$ denotes the ideal of morphism of
$\mathcal{C}$ which factor through a direct sum of copies of $T[1]$
and $[1]$ denote the shift functor. Let $T_1, \cdots, T_n$ be the
pairwise non-isomorphic indecomposable direct summands of $T$ and
$S_i$ be the simple tops of projective $B$-modules $P_i=FT_i$ for
$i=1, \cdots, n$. In \cite{Palu1}, the author generalized the
Caldero-Chapoton map (\cite{CC}) as follows. Define a map
$$X_?: \mathrm{obj}(\mc)\rightarrow \bbq(x_1, \cdots x_n)
$$ by mapping any $M\in \mathrm{obj}(\mathcal{C})$ to
$$X^{T}_{M}=\underline{x}^{-\mathrm{coind}\ M}\sum_{\ue\in K_{0}(\mathrm{mod}\ B)}\chi(\mathrm{Gr}_{\ue}FM)\prod_{i=1}^{n}x_{i}^{\langle\underline{s}_{i}, \ue\rangle_{a}}$$
where $\underline{s}_i=\underline{\mathrm{dim}}S_i$ for $i=1,
\cdots, n$ and we refer to \cite{Palu1} for the definitions of the
coindex $\mathrm{coind}\ M$ of $M$ and the antisymmetic bilinear
form $\langle-, -\rangle_{a}$ on $K_{0}(\mathrm{mod}\ B)$.

Let $L$ and $M$ be objects in $\mc$. Given a subset $S\subseteq
\mathrm{obj}(\mc)$, we denote by $\mc(L, M[1])_{S}$ the set of
morphisms in $\mc(L, M[1])$ with the middle terms belonging to $S$.
For any object $Y$ in $\mc$, we denote by $\langle Y\rangle$ the set
$$\{Y'\in \mathrm{obj}(\mc)|\mathrm{coind}\ Y'=\mathrm{coind}\ Y,
\chi(\mathrm{Gr}_{\ue}FY')=\chi(\mathrm{Gr}_{\ue}FY) \mbox{ for any
} \ue\in K_{0}(\mathrm{mod}\ B)$$ $$Y' \mbox{ is the middle term of
some morphism in } \mc(L, M[1]) \mbox{ or } \mc(M, L[1])\}.$$ If the
cylinders over the morphisms $L\rightarrow M[1]$ and $M\rightarrow
L[1]$ are constructible with respect to $T$ (see \cite[Section
1.3]{Palu2} for definition), then there exists a finite subset
$\mathcal{Y}$ of the set of middle terms of morphisms in $\mc(L,
M[1])$ or $\mc(M, L[1])$ such that the subsets $\mc(L,
M[1])_{\langle Y\rangle}$ (resp. $\mc(M, L[1])_{\langle Y\rangle}$)
are constructible in $\mc(L,M[1])$ (resp. $\mc(M, L[1])$) for $Y\in
\mathcal{Y}$ and there are finite stratifications
$$
\mc(L, M[1])=\bigsqcup_{Y\in \mathcal{Y}}\mc(L, M[1])_{\langle
Y\rangle} \mbox{ and } \mc(M, L[1])=\bigsqcup_{Y\in
\mathcal{Y}}\mc(M, L[1])_{\langle Y\rangle}
$$
(\cite[Proposition 9]{Palu2}).
\begin{theorem}\cite[Theorem 1]{Palu2}\label{Pa}
With the above notation, assume that for any $L, M\in
\mathrm{obj}(\mc)$,  the cylinders over the morphisms $L\rightarrow
M[1]$ and $M\rightarrow L[1]$ are constructible with respect to $T$.
Then we have
$$\chi(\mathbb{P}\mathcal{C}(L,M[1]))X^{T}_{L}X^{T}_{M}=\sum_{Y\in \mathcal{Y}}(\chi(\mathbb{P}\mathcal{C}(L,M[1])_{\langle Y\rangle})+\chi(\mathbb{P}\mathcal{C}(M,L[1])_{\langle Y\rangle}))X^{T}_{Y}.$$
\end{theorem}
The formula in the theorem is called the cluster multiplication
formula for 2-Calabi-Yau categories with cluster tilting objects. By
definition, we have $$\mathcal{C}(L,M[1]) \cong
\mathcal{C}/(T[1])(L,M[1])\oplus (T[1])(L, M[1]).$$ Hence, the
spaces $\mathcal{C}/(T[1])(L,M[1])$ and $(T[1])(L, M[1])$ can be
viewed as the subspaces of $\mc(L, M[1])$ and then we set
$$\mathcal{C}/(T[1])(L,M[1])_{\langle Y\rangle}:=
\mathcal{C}(L,M[1])_{\langle Y\rangle}\cap
\mathcal{C}/(T[1])(L,M[1])$$ and $$(T[1])(L,M[1])_{\langle
Y\rangle}:= \mathcal{C}(L,M[1])_{\langle Y\rangle}\cap
(T[1])(L,M[1]).$$  Then it is clear that
$\mathcal{C}/(T[1])(L,M[1])_{\langle Y\rangle}$ (resp.
$(T[1])(L,M[1])_{\langle Y\rangle}$) are constructible subsets of
$\mathcal{C}/(T[1])(L,M[1])$ (resp. $(T[1])(L,M[1])$) and there are
finite stratifications
$$
\mc/(T[1])(L, M[1])=\bigsqcup_{Y\in \mathcal{Y}}\mc/(T[1])(L,
M[1])_{\langle Y\rangle}, \mc/(T[1])(M, L[1])=\bigsqcup_{Y\in
\mathcal{Y}}\mc/(T[1])(M, L[1])_{\langle Y\rangle}$$ and $$(T[1])(L,
M[1])=\bigsqcup_{Y\in \mathcal{Y}}(T[1])(L, M[1])_{\langle
Y\rangle}, (T[1])(M, L[1])=\bigsqcup_{Y\in \mathcal{Y}}(T[1])(M,
L[1])_{\langle Y\rangle}.
$$
We will prove its following variant which is a generalization of the
cluster multiplication formula of any type in \cite{XiaoXu} and
\cite{Xu}(see Theorem \ref{X-X}).
\begin{theorem}\label{DX}
With the notation and assumption of Theorem \ref{Pa}, we have
$$\hspace{-3cm}\chi(\mathbb{P}\mathcal{C}/(T[1])(L,M[1]))X^{T}_{L}X^{T}_{M}$$$$=\sum_{Y\in \mathcal{Y}}(\chi(\mathbb{P}\mathcal{C}/(T[1])(L,M[1])_{\langle Y\rangle})
+\chi(\mathbb{P}(T[1])(M,L[1])_{\langle Y\rangle}))X^{T}_{Y}$$ where
$(T[1])(M,L[1])$ denotes the subset of $\mathcal{C}(M,L[1])$
consisting of the morphisms in
 which factor through a direct sum of copies of
$T[1].$
\end{theorem}

\nd \emph{Proof of Theorem \ref{Pa} by Theorem \ref{DX}}: Since
$\mc$ is 2-Calabi-Yau, we have
$$\mathcal{C}/(T[1])(M,L[1])=D(T[1])(L,M[1])$$ ([Palu1, Lemma 10]).
Then we obtain a decompositions of $\mathcal{C}(L,M[1])$:
\begin{eqnarray}
   \mathcal{C}(L,M[1]) &\cong & \mathcal{C}/(T[1])(L,M[1])\oplus \mathcal{C}/(T[1])(M,L[1]).\nonumber
\end{eqnarray}
Hence, we have
\begin{equation}\label{e-1}
\chi(\mathcal{C}(L,M[1]))=\chi(\mathcal{C}/(T[1])(L,M[1]))+\chi(\mathcal{C}/(T[1])(M,L[1])).
\end{equation} We define an action of $\bbc^*$ on $\mathbb{P}\mc(L, M[1])$ by
$$
t.\mathbb{P}(\varepsilon_1,\varepsilon_2)=\mathbb{P}(t\varepsilon_1,t^2\varepsilon_2)
$$
for $t\in \bbc^*$ and $\mathbb{P}(\varepsilon_1,\varepsilon_2)\in
\mathbb{P}\mathcal{C}(L,M[1]).$ It is easily know that
$\mathbb{P}(\varepsilon_1,\varepsilon_2)$ is stable under this
action if and only if either $\varepsilon_1$ or $\varepsilon_2$
vanishes. Hence, for $Y\in \mathcal{Y}$, we have
\begin{equation}\label{e-2}
\chi(\mathbb{P}\mathcal{C}(L,M[1])_{\langle
Y\rangle})=\chi(\mathbb{P}\mathcal{C}/(T[1])(L,M[1])_{\langle
Y\rangle})+\chi(\mathbb{P}\mathcal{C}/(T[1])(M, L[1])_{\langle
Y\rangle}).
\end{equation}
Dually, we also have the decomposition
\begin{eqnarray}
   \mathcal{C}(M, L[1]) &\cong & (T[1])(L,M[1])\oplus (T[1])(M,L[1]).\nonumber
\end{eqnarray}
In the same way as deducing the equation \eqref{e-2}, we obtain that
\begin{equation}\label{e-3}
\chi(\mathbb{P}\mathcal{C}(M, L[1])_{\langle
Y\rangle})=\chi(\mathbb{P}(T[1])(L,M[1])_{\langle
Y\rangle})+\chi(\mathbb{P}(T[1])(M, L[1])_{\langle Y\rangle})
\end{equation}
for $Y\in \mathcal{Y}$.

By the equations \eqref{e-1}, \eqref{e-2} and \eqref{e-3}, it is
easy to prove Theorem \ref{Pa} by using Theorem \ref{DX}.
\hfill$\square$

 Hence, the formula in Theorem \ref{DX} is a refinement of the
formula in Theorem \ref{Pa}.
\section{Proof of Theorem \ref{DX}}
Since the proof of Theorem \ref{DX} is similar to  the proof of
Theorem \ref{Pa} in \cite{Palu2}, we just sketch the proof. Set
$$
\Sigma_{1}=\sum_{Y\in
\mathcal{Y}}\chi(\mathbb{P}\mathcal{C}/(T[1])(L,M[1])_{\langle
Y\rangle})X^{T}_{Y}.
$$
By definition, it is equal to
$$
\sum_{Y\in
\mathcal{Y}}\sum_{\underline{g}}\chi(\mathbb{P}\mathcal{C}/(T[1])(L,M[1])_{\langle
Y\rangle})\chi(\mathrm{Gr}_{\underline{g}}FY)\underline{x}^{-\mathrm{coind}\
Y}\prod_{i=1}^{n}x_{i}^{\langle\underline{s}_{i},
\underline{g}\rangle_{a}}.
$$

Given $\varepsilon\in \mathcal{C}/(T[1])(L,M[1])_{\langle
Y\rangle}$, then it induces a triangle $M\xrightarrow{i}
Y'\xrightarrow{p} L\xrightarrow{\varepsilon} M[1].$ Set $\Delta:
=\underline{\mathrm{dim}}\ FL+\underline{\mathrm{dim}}\ FM.$ Denote
by $W_{LM}^{Y}(\ue,\underline{f},\underline{g})$ the set
$$\hspace{-3cm}\{(\mathbb{P}\varepsilon, E)\mid \mathbb{P}\varepsilon\in
\mathbb{P}\mathcal{C}/(T[1])(L,M[1])_{\langle Y\rangle}, E\in
\mathrm{Gr}_{\underline{g}}(FY'),
$$$$\hspace{2cm}\underline{\mathrm{dim}}Fp(E)=\underline{e},
\underline{\mathrm{dim}}(Fi)^{-1}(E)=\underline{f}\}.$$ It is a
constructible subset of  the set $
\mathbb{P}\mathcal{C}/(T[1])(L,M[1])\times
\bigsqcup_{\underline{d}\leq \Delta}\prod_{i=1}^{n}
\mathrm{Gr}_{g_i}(k^{d_i})$ where $\ud=(d_i)_{i=1}^n$ and
$\underline{g}=(g_i)_{i=1}^n$(compared to \cite[Lemma 17]{Palu2}).
We set
$$W_{LM}^{Y}(\underline{g})=\bigsqcup_{\underline{e}\leq \underline{\mathrm{dim}}\ FL, \underline{f}\leq \underline{\mathrm{dim}}\ FM}W_{LM}^{Y}(\underline{e},\underline{f},\underline{g})
\quad \mbox{and}\quad W_{LM}^{Y}(\underline{e},
\underline{f})=\bigsqcup_{\underline{g}\leq
\Delta}W_{LM}^{Y}(\underline{e}, \underline{f}, \underline{g}).$$

Consider the projection
$$\pi: W_{LM}^{Y}(\underline{g})\longrightarrow \mathbb{P}\mathcal{C}/(T[1])(L,M[1])_{\langle Y\rangle}$$
It is obvious that
$\pi^{-1}(\mathbb{P}\varepsilon)=\{\mathbb{P}\varepsilon\}\times
\mathrm{Gr}_{\underline{g}}FY'$ for any $\mathbb{P}\varepsilon\in
\mathbb{P}\mathcal{C}/(T[1])(L,M[1])_{\langle Y\rangle}$. Since
$\chi(\mathrm{Gr}_{\underline{g}}FY')=\chi(\mathrm{Gr}_{\underline{g}}FY)$,
we have
$$\chi(W_{LM}^{Y}(\underline{g}))=\chi(\mathbb{P}\mathcal{C}/(T[1])(L,M[1])_{\langle
Y\rangle})\chi(\mathrm{Gr}_{\underline{g}}FY).$$ Hence, we obtain
\begin{eqnarray}
   \Sigma_1 &=& \sum_{\underline{g}, Y\in \mathcal{Y}}\chi(W_{LM}^{Y}(\underline{g}))\underline{x}^{-\mathrm{coind}\
Y}\prod_{i=1}^{n}x_{i}^{\langle\underline{s}_{i}, \underline{g}\rangle_{a}}\nonumber\\
  &=& \sum_{\underline{e}, \underline{f}, \underline{g}, Y\in \mathcal{Y}}\chi(W_{LM}^{Y}(\underline{e}, \underline{f}, \underline{g}))\underline{x}^{-\mathrm{coind}\ Y}\prod_{i=1}^{n}x_{i}^{\langle\underline{s}_{i}, \underline{g}\rangle_{a}}.\nonumber
\end{eqnarray}
The following equality in $K_0(\mathrm{mod} B)$ (\cite[Lemma
16]{Palu1})
$$\sum_{i=1}^{n}\langle\underline{s}_{i}, \underline{e}+\underline{f}\rangle_{a}[P_i]-\mathrm{coind}\ (L\oplus M)=\sum_{i=1}^{n}\langle\underline{s}_{i}, \underline{g}\rangle_{a}[P_i]-\mathrm{coind}\ Y$$
implies that
$$
\Sigma_1=\sum_{\underline{e} ,\underline{f}, Y\in
\mathcal{Y}}\chi(W_{LM}^{Y}(\underline{e},
\underline{f}))\underline{x}^{-\mathrm{coind}\ (L\oplus
M)}\prod_{i=1}^{n}x_{i}^{\langle\underline{s}_{i},
\underline{e}+\underline{f}\rangle_{a}}.
$$
 In order to make the connection between
$\Sigma_1$ and the left side of the equation in Theorem \ref{DX}, it
is natural to consider the following (constructible) map
$$\psi_{\ue, \underline{f}}: \coprod_{Y\in \mathcal{Y}}W_{LM}^{Y}(\underline{e}, \underline{f})\longrightarrow \mathbb{P}\mathcal{C}/(T[1])(L,M[1])\times \mathrm{Gr}_{\underline{e}}FL\times \mathrm{Gr}_{\underline{f}}FM$$
defined by mapping $(\mathbb{P}\varepsilon, E)$ to
$(\mathbb{P}\varepsilon, Fp(E), (Fi)^{-1}(E)).$ Since the fibre of
any point in $\mathrm{Im} \psi_{\underline{e}, \underline{f}}$ are
affine spaces by \cite{CC}, $\sum_{Y\in
\mathcal{Y}}\chi(W_{LM}^{Y}(\underline{e},
\underline{f}))=\chi(\mathrm{Im} \psi_{\underline{e},
\underline{f}}).$ Then we have
$$
\Sigma_1=\sum_{\underline{e},
\underline{f}}\chi(\mathrm{Im}\psi_{\ue,
\underline{f}})\underline{x}^{-\mathrm{coind}\ (L\oplus
M)}\prod_{i=1}^{n}x_{i}^{\langle\underline{s}_{i},\underline{e}+\underline{f}\rangle_{a}}.
$$
Set $L(\underline{e},
\underline{f})=\mathbb{P}\mathcal{C}/(T[1])(L,M[1])\times
\mathrm{Gr}_{\underline{e}}FL\times \mathrm{Gr}_{\underline{f}}FM.$
We denote by $L_2(\ue, \underline{f})$ the complement of
$\mathrm{Im}\psi_{\ue, \underline{f}}$ in $L(\underline{e},
\underline{f})$.

Dually, we set
$$
\Sigma_{2}=\sum_{Y\in \mathcal{Y}}\chi(\mathbb{P}(T[1])(M,
L[1])_{\langle Y\rangle})X^{T}_{Y}.
$$
Given $\varepsilon\in (T[1])(M, L[1])_{\langle Y\rangle}$, then it
induces a triangle $L\xrightarrow{i} Y'\xrightarrow{p}
M\xrightarrow{\varepsilon} L[1].$ Denote by
$W_{ML}^{Y}(\underline{f},\ue,\underline{g})$ the set
$$\hspace{-3cm}\{(\mathbb{P}\varepsilon, E)\mid \mathbb{P}\varepsilon\in
\mathbb{P}(T[1])(M, L[1])_{\langle Y\rangle}, E\in
\mathrm{Gr}_{\underline{g}}(FY'),
$$$$\hspace{2cm}\underline{\mathrm{dim}}Fp(E)=\underline{e},
\underline{\mathrm{dim}}(Fi)^{-1}(E)=\underline{f}\}.$$ Similarly,
we set
$$W_{ML}^{Y}(\underline{g})=\bigsqcup_{\underline{e}\leq \underline{\mathrm{dim}}\ FL, \underline{f}\leq \underline{\mathrm{dim}}\ FM}W_{ML}^{Y}(\underline{f},\underline{e},\underline{g})
\quad \mbox{and}\quad W_{ML}^{Y}(\underline{f},
\underline{e})=\bigsqcup_{\underline{g}\leq
\Delta}W_{ML}^{Y}(\underline{f}, \underline{e}, \underline{g}).$$
Then we have
$$
\Sigma_2=\sum_{\underline{e}, \underline{f}, \underline{g}, Y\in
\mathcal{Y}}\chi(W_{ML}^{Y}(\underline{f},\underline{e},\underline{g}))\underline{x}^{-\mathrm{coind}\
  Y}\prod_{i=1}^{n}x_{i}^{\langle\underline{s}_{i},\underline{g}\rangle_{a}}.
$$

Since $\mathcal{C}$ is a 2-Calabi-Yau category and by [Palu1, Lemma
10], there is a non degenerate bifunctorial pairing
$$\phi:\mathcal{C}/(T)(L[-1],M) \times (T[1])(M,L[1])\longrightarrow
k.$$ Let $V\xrightarrow{i_{V}} L$ and $U\xrightarrow{i_{U}} M$ be
two morphisms such that $F(i_V)$ and $F(i_U)$ are monomorphisms.
Similar to \cite{Palu1}, we define the map
$$\hspace{-3cm}\alpha: \mathcal{C}(L[-1],U)\oplus \mathcal{C}/(T)(L[-1],M)\longrightarrow$$$$\hspace{3cm} \mathcal{C}/(T)(V[-1],U)\oplus \mathcal{C}(V[-1],M)\oplus \mathcal{C}/(T[1])(L[-1],M)$$
by mapping $(a,b)$ to $
(ai_{V}[-1],i_{U}ai_{V}[-1]-bi_{V}[-1],i_{U}a-b).$ Its dual map is
$$\alpha':(T[1])(U,V[1])\oplus \mathcal{C}(M,V[1])\oplus
\mathcal{C}/(T[2])(M,L[1])\longrightarrow \mathcal{C}(U,L[1])\oplus
(T[1])(M,L[1])$$
$$(a,b,c)\longrightarrow (i_{V}[1]a+ci_{U}+i_{V}[1]bi_{U},-c-i_{V}[1]b)$$
where $V\in \mathrm{Gr}_{\underline{e}}FL, U\in
\mathrm{Gr}_{\underline{f}}FM$. Denote by $C_{\underline{e},
\underline{f}}(Y,\underline{g})$ the set
$$\{((\mathbb{P}\varepsilon, V, U),(\mathbb{P}\eta, E))\in
L_2(\underline{e},\underline{f})\times W_{ML}^{Y}(\underline{g}) |
\phi(\varepsilon[-1],\eta)\neq 0,$$
$$
(Fi)^{-1}E=V, (Fp)(E)=U, \text{and}\ i, p\ \text{are given by}\
\eta\}$$ and $$C_{\underline{e}, \underline{f}}=\bigsqcup_{ Y\in
\mathcal{Y}, \underline{g}\in K_0(\mathrm{mod}
B)}C_{\underline{e},\underline{f}}(Y,\underline{g}).$$ The following
propositions also hold for $\alpha$ and $\alpha'$ defined here. We
refer to \cite{Palu2} for the proofs.
\begin{Prop}\cite[Proposition 3]{CK2005}, \cite[Proposition 19]{Palu2}\label{DM1}
With the above notations, the following assertions are equivalent:\\
(i)\ $(\mathbb{P}\varepsilon,V,U)\in L_2(\underline{e},\underline{f}).$\\
(ii)\ $\varepsilon[-1]$ is not orthogonal to $(T[1])(M,L[1])\cap
\mathrm{Im}
\alpha'.$\\
(iii)\ There is an $\eta\in (T[1])(M,L[1])$ such that
$\phi(\varepsilon[-1],\eta)\neq 0$ and such that if
$$L\xrightarrow{i} N\xrightarrow{p} M\xrightarrow{\varepsilon}
L[1]$$ is a triangle in $\mathcal{C},$ then there exists the
submodule $E$ of $FN$ such that $(Fi)^{-1}E=V, (Fp)(E)=U.$
\end{Prop}

\begin{Prop}\cite[Proposition 4]{CK2005}, \cite[Proposition 20]{Palu2}\label{DM2}

\nd (a)\ The projection $C_{\underline{e},\underline{f}}\xrightarrow{\pi_1} L_2(\underline{e},\underline{f})$ is surjective and the Euler characteristic of any fibre of $pi_1$ is $1$.\\
(b)\ The projection $C_{\underline{e},\underline{f}}(Y,\underline{g})\xrightarrow{\pi_2}W_{ML}^{Y}(\underline{f},\underline{e},\underline{g})$ is surjective and its fibres are affine spaces.\\
(c)\ If $C_{\underline{e},\underline{f}}(Y,\underline{g})$ is not
empty, then we have
$$\sum_{i=1}^{n}\langle\underline{s}_{i},\underline{e}+\underline{f}\rangle_{a}[P_i]-\mathrm{coind}\ (L\oplus M)=\sum_{i=1}^{n}\langle\underline{s}_{i},\underline{g}\rangle_{a}[P_i]-\mathrm{coind}\ Y.$$
\end{Prop}

As a consequence, we have the following corollary.
\begin{Cor}\label{DM3}
$\chi(C_{\underline{e},\underline{f}})=\chi(L_2(\underline{e},\underline{f}))$
and
$\chi(C_{\underline{e},\underline{f}}(Y,\underline{g}))=\chi(W_{ML}^{Y}(\underline{f},\underline{e},\underline{g})).$
\end{Cor}
Then by Proposition \ref{DM1}, \ref{DM2} and Corollary \ref{DM3}, we
obtain
$$
 \Sigma_2 =\sum_{\underline{e},\underline{f}}\chi(L_{2}(\underline{e},\underline{f}))\underline{x}^{-\mathrm{coind}\ (L\oplus
 M)}\prod_{i=1}^{n}x_{i}^{<\underline{s}_{i},\underline{e}+\underline{f}>_{a}}.
$$
Then, we have
$$\Sigma_1+\Sigma_2=\sum_{\underline{e},\underline{f}}\chi(L(\underline{e},\underline{f}))\underline{x}^{-\mathrm{coind}\
(L\oplus
 M)}\prod_{i=1}^{n}x_{i}^{<\underline{s}_{i},\underline{e}+\underline{f}>_{a}}$$
 $$\hspace{-0.5cm}=\chi(\mathbb{P}\mc/(T[1])(L, M[1]))X^T_LX^T_M.$$ This completes the
proof of Theorem \ref{DX}.

\section{Application to acyclic quivers}
In this section, we assume $\mathcal{C}$ is the cluster category of
an acyclic quiver $Q$ and $T=kQ.$ Thus the cluster character defined
in \cite{Palu1} is exactly the Caldero-Chapoton map. So we simply
write $X_{?}$ instead of $X_{?}^{T}.$
\begin{Prop}\label{DM5}
Assume $L,M$ are $kQ-$modules and $M$ has no projective modules as
its direct summand, then we have the following isomorphisms between
vector spaces
$$\mathrm{Hom}_{kQ}(L,\tau M)\cong \mathrm{Ext}^{1}_{kQ}(M,L)\cong (T[1])(M,L[1])\cong \mathcal{C}/(T[1])(L,M[1]).$$
\end{Prop}
\begin{proof}
The first isomorphism is an application of the Auslander-Reiten
formula. The third isomorphism follows from \cite[Lemma 10]{Palu1}.
It is enough to prove $\mathrm{Ext}^{1}_{kQ}(M,L)\cong
(T[1])(M,L[1])$. Consider the projective resolution of $M$ short
exact sequence:
$$0\longrightarrow P_0\longrightarrow P_1\longrightarrow M\longrightarrow 0$$
such that $kQ\subseteq P_0$. Applying the functor
$\mathrm{Hom}_{kQ}(-, L)$ and $\mc(-, L)$ on it, we obtain  long
exact sequences
$$0\longrightarrow \mathrm{Hom}_{kQ}(M,L)\longrightarrow
\mathrm{Hom}_{kQ}(P_1,L)\longrightarrow
\mathrm{Hom}_{kQ}(P_0,L)\xrightarrow{f}
\mathrm{Ext}^{1}_{kQ}(M,L)\longrightarrow 0$$ and
$$\cdots\longrightarrow\mathcal{C}(P_1,L)\longrightarrow \mathcal{C}(P_0,L)\xrightarrow{g}
\mathcal{C}(M[-1],L)\longrightarrow \cdots.
$$
Note that $\mathrm{Hom}_{kQ}(P_1,L)\cong \mathcal{C}(P_1,L)$ and
$\mathrm{Hom}_{kQ}(P_0,L)\cong \mathcal{C}(P_0,L)$ by \cite{BMRRT},
thus there is a natural mono map $i:
\mathrm{Ext}^{1}_{kQ}(M,L)\longrightarrow \mathcal{C}(M[-1],L)$
induces $if\simeq g.$ Since $f$ is surjective, we have
$$\mathrm{Ext}^{1}_{kQ}(M,L)\cong \mathrm{Im}\ g\cong (P_0)(M[-1],L)\cong (P_0[1])(M,L[1])=(T[1])(M,L[1]).$$
\end{proof}

Using Theorem \ref{DX} and Proposition \ref{DM5}, we can obtain  the
 following cluster multiplication
 formulas for acyclic quivers of any type in the context of module categories.
\begin{theorem}\label{X-X}(\cite{XiaoXu}, \cite[Theorem 4.1]{Xu})
Let $kQ$ be an acyclic quiver. Then

\nd(1)\ For any $kQ-$modules $L,M$, and $M$ has no projective
modules as its direct summands,  we have
$$\mathrm{dim}_{k}\mathrm{Ext}^{1}_{kQ}(M,L)X_{L}X_{M}=\sum_{Y\in \mathrm{R(\ue)}}(\chi(\mathbb{P}\mathrm{Ext}^{1}_{kQ}(M,L)_{\langle Y\rangle})
+\chi(\mathbb{P}\mathrm{Hom}_{kQ}(L,\tau M)_{\langle
Y\rangle}))X_{Y}$$ where $\ue=\underline{\mathrm{dim}}
M+\underline{\mathrm{dim}} N.$

\nd(2)\ For any $kQ-$module $M$, and projective module $P$, we have
$$\mathrm{dim}_{k} \mathrm{Hom}_{kQ}(P,M)X_{M}X_{P[1]}=\sum_{Y\in \mathcal{S}}(\chi(\mathbb{P}\mathrm{Hom}_{kQ}(M,I)_{\langle Y\rangle})
+\chi(\mathbb{P}\mathrm{Hom}_{kQ}(P,M)_{\langle Y'\rangle}))X_{Y}$$
where $I=\mathrm{DHom}_{kQ}(P, kQ).$
\end{theorem}
Here, $\mathrm{R}(\ue)$ and $\mathcal{S}$ are some finite sets
defined as the finite set $\mathcal{Y}$ in Theorem \ref{Pa}, we
refer to \cite{Xu} for details.
\begin{proof}
(1)\ By Proposition \ref{DM5}, the first formula in Theorem \ref{DX}
and the fact that
$\chi(\mathbb{P}\mathrm{Ext}^{1}_{kQ}(M,L))=\mathrm{dim}_{k}\mathrm{Ext}^{1}_{kQ}(M,L).$\\
(2)\ We know $I=P[2]\in \mathrm{obj}(\mathcal{C})$ and then
$$\mathcal{C}/(T[1])(M,P[2])_{\langle Y\rangle}\cong \mathcal{C}/(T[1])(M,I)_{\langle Y\rangle}\cong \mathrm{Hom}_{kQ}(M,I)_{\langle Y\rangle}.$$
We also have
$$D\mathcal{C}/(T[1])(M,P[2])\cong (T[1])(P[1],M[1])\cong (T)(P,M)\cong \mathrm{Hom}_{kQ}(P,M).$$ Therefore
the proof follows from Theorem \ref{DX} and the fact that
$\chi(\mathbb{P}\mathrm{Hom}_{kQ}(P,M))=\mathrm{dim}_{k}\mathrm{Hom}_{kQ}(P,M)$.
\end{proof}
In the same way as in the proof of Theorem \ref{Pa} by Theorem
\ref{DX}, we obtain a `unified' version of the formulas in Theorem
\ref{X-X} in the context of cluster categories. It is clear that
Theorem 2.1 is a generalization of Theorem \ref{unified}.
\begin{theorem}\label{unified}
Let $kQ$ be an acyclic quiver and $\mc$ be the cluster category of
$kQ$. For any $M, N\in \mathrm{obj} (\mc)$, if
$\mathrm{Ext}^1_{\mc}(M, N):=\mc(M, N[1])\neq 0$, we have
$$
\chi(\mathbb{P}\mathrm{Ext}^1_{\mc}(M, N))X_MX_N=\sum_{Y\in
\mathcal{Y}}(\chi(\mathbb{P}\mathrm{Ext}^1_{\mc}(M, N)_{\langle
Y\rangle})+\chi(\mathbb{P}\mathrm{Ext}^1_{\mc}(M, N)_{\langle
Y\rangle}))X_Y.
$$
\end{theorem}
If either $\mathrm{Ext}^1_{kQ}(M, L)$ or $\mathrm{Ext}^1_{kQ}(L, M)$
vanishes, then Theorem \ref{X-X} coincides with Theorem
\ref{unified} (see \cite{Hubery2005}). In general, it is possible
that both of two extension spaces are not zero.  For example,
consider the affine quiver $Q=\widetilde{D}_4$ with a fixed
orientation. Let $\mc$ be the cluster category of type
$\widetilde{D}_4$. The Auslander-Reiten quiver of $\widetilde{D}_4$
contains three nonhomogeneous tubes, denoted by $\mathcal{T}_0,
\mathcal{T}_1, \mathcal{T}_{\infty}$. Assume that $E_1, E_2$ are two
regular simple modules in $\mathcal{T}_{0}$. Then we have
$$\mathrm{dim}_k\mathrm{Ext}^1_{Q}(E_1, E_2)=\mathrm{dim}_k \mathrm{Ext}^1_Q(E_2, E_1)=1, \quad \mathrm{dim}_{k}\mathrm{Ext}^1_{\mc}(E_1, E_2)=2.$$
By Theorem \ref{X-X}, we obtain
$$
X_{E_1}X_{E_2}=X_{E_1[2]}+1.
$$
where $E_1[2]$ denotes the extension of $E_1$ by $E_2.$ By Theorem
\ref{unified}, we obtain
$$
2X_{E_1}X_{E_2}=2X_{E_1[2]}+2.
$$

\section*{Acknowledgements}
The main idea of this paper comes from a sequence of discussions
 with our teacher Jie Xiao. We are grateful
 to Professor Jie Xiao for his  fruitful ideas. The second author thanks Claus Michael Ringel for his help and
hospitality while staying at University of Bielefeld as an Alexander
von Humboldt Foundation fellow. The authors thank Y. Palu for his
kind suggestions.

\end{document}